\documentclass[10pt, oneside]{amsart}   	
\usepackage{geometry}                		
\usepackage{graphicx}				
\usepackage{amssymb}
\usepackage{amsmath}
\usepackage{amsthm}
\usepackage{hyperref}
\usepackage{bbm}
\usepackage{color,soul}

\newcommand{\bbN}{\mathbb{N}}

\newcommand{\bbR}{\mathbb{R}}

\newcommand{\bbZ}{\mathbb{Z}}

\newcommand{\eps}{\varepsilon}

\DeclareMathOperator{\tr}{tr}
\DeclareMathOperator{\sgn}{sgn}

\newtheorem{theorem}{Theorem}[section]
\newtheorem{lemma}[theorem]{Lemma}
\newtheorem{corollary}[theorem]{Corollary}
\newtheorem*{theorem*}{Theorem}

\numberwithin{equation}{section}

\title{Low Regularity Conservation Laws for the Benjamin-Ono Equation}
\author{Blaine Talbut}
\date{}							

\begin{document}

\maketitle

\begin{abstract}
We obtain conservation laws at negative regularity for the Benjamin-Ono equation on the line and on the circle. These conserved quantities control the $H^s$ norm of the solution for $-\frac{1}{2} < s < 0$. The conservation laws are obtained from a study of the perturbation determinant associated to the Lax pair of the equation.
\end{abstract}

\section{Introduction}

We study real-valued solutions to the Benjamin-Ono equation
\begin{equation} \label{BO} \tag{BO}
\frac{d}{dt} q = - H q'' + 2 q q'
\end{equation}
on the line $\bbR$ and the circle $\bbR / \bbZ$, where the Hilbert transform $H$ is defined in either setting by 
\[
\widehat{Hf}(\xi) = - i \sgn(\xi) \hat{f}(\xi).
\]
This equation is a model for the propagation of long internal waves. For a recent review of the literature on \eqref{BO}, see \cite{saut-review}. The equation \eqref{BO} is known to be completely integrable and to enjoy an infinite hierarchy of conservation laws which control the $H^{k/2}$ norms of the solution for $k = 0, 1, \ldots$ (see \cite{fokas-ablowitz}, \cite{coifman-wickerhauser}). The well-posedness of the Cauchy problem for \eqref{BO} has been well-studied on both the line (\cite{koch-tzvetkov}, \cite{kenig-koenig}, \cite{tao}, \cite{burq-planchon}) and the circle (\cite{molinet-ribaud}, \cite{molinet-1/2}). On both spaces, the lowest regularity for which global well-posedness is known in $H^s$ is $s = 0$ (see \cite{ionescu-kenig}, \cite{molinet-gwp}, \cite{molinet-pilot}, \cite{ifrim-tataru}). The equation \eqref{BO} is also known to be well-posed in the category of $C^0_t H^3_x \cap C^1_t H^1_x$ classical solutions (see \cite{iorio}, \cite{Saut}); note that a classical $C^0_t H^3_x$ solution is automatically $C^1_t H^1_x$ because it solves \eqref{BO}. For these results at nonnegative regularity, global well-posedness can be deduced from local well-posedness and the aforementioned hierarchy of conservation laws.

The equation \eqref{BO} also enjoys a scaling symmetry, to wit
\[
q \mapsto \lambda q( \lambda^2 t, \lambda x) .
\]
This symmetry leaves $\| q \|_{\dot{H}^{-1/2}(\bbR)}$ invariant. This suggests that at least below the critical regularity $s = -1/2$ we should expect \eqref{BO} to be ill-posed. Well-posedness on the line in the regime $- \frac{1}{2} \leq s < 0$ appears to be an open question, and we hope to apply the conservation laws obtained in this paper to a future study of this problem. On the other hand, on the circle, for all $s < 0$ the Cauchy problem is known to be ill-posed in the sense that the data-to-solution map fails to be pointwise continuous; see \cite{molinet-ill-posedness}. Our results show that, nevertheless, norm blowup does not occur on the circle for regularities $s > - \frac{1}{2}$. This leaves open the possibility that after some suitable renormalization of the solutions, one can recover well-posedness on the circle, as was done for the cubic Wick-ordered NLS on the circle in \cite{oh-sulem}.

The equation \eqref{BO} is related to the Korteweg--de Vries equation
\begin{equation} \label{KdV} \tag{KdV}
\frac{d}{dt} q = -q''' + 6qq'
\end{equation}
by way of the Intermediate Long Wave equation; \eqref{BO} is formally obtained from the Intermediate Long Wave equation in the deep water limit, while \eqref{KdV} arises from the shallow water limit. For details, see \cite{ablowitz-clarkson}. Using the integrable structure of \eqref{KdV} and in particular the Lax pair, Killip, Vi\c{s}an, and Zhang \cite{killip-visan-zhang} obtained conservation laws which govern the $H^s$ norm of the solution for $s \geq -1$. These same conservation laws were employed in \cite{kdv-gwp} to obtain global well-posedness of \eqref{KdV} in the space $H^{-1}$.

In this note, we follow the method of \cite{killip-visan-zhang} to obtain low-regularity conservation laws for \eqref{BO}. Our principal result is the following:
\begin{theorem*}
Let $q$ be a classical solution to \eqref{BO} on the line or the circle and let $- \frac{1}{2} < s < 0$, $1 \leq r \leq \infty$. Then
\[
(1 + \| q(0)\|^{\frac{2}{1+2s}}_{B^{s,2}_r})^{s} \sup_{t \in \bbR} \| q(t) \|_{B^{s,2}_r} \lesssim_{s,r} \| q(0) \|_{B^{s,2}_r}
	\lesssim_{s,r} (1 + \|q(0)\|^{\frac{2}{1+2s}}_{B^{s,2}_r})^{-s} \inf_{t \in \bbR} \| q(t)\|_{B^{s,2}_r} .
\]
The particular case of $r=2$ is equivalent to the conservation of the Sobolev norm:
\[
\sup_{t \in \bbR} (1 + \| q(0)\|^{\frac{2}{1+2s}}_{H^s})^{s -} \| q(t) \|_{H^s} \lesssim_{s,r} \| q(0) \|_{H^s}
	\lesssim_{s,r} \inf_{t \in \bbR} (1 + \|q(0)\|^{\frac{2}{1+2s}}_{H^s})^{-s+} \| q(t)\|_{H^s} .
\]
\end{theorem*}
This will be proved as Theorem \ref{main-theorem}. See section 3 for the definition of the Besov norms $\| f \|_{B^{s,2}_r}$.

Let us review the method of \cite{killip-visan-zhang} as it applies to our problem. The first thing to note is that \eqref{BO} has a Lax pair. We proceed formally, leaving aside considerations of boundedness until we have identified the objects of our study. We follow \cite{wu} in presenting the Lax pair as it decomposes along the Hardy spaces $H^{\pm}$ of $L^2$ functions whose Fourier transforms are supported on positive and negative modes, respectively. On the line, 
\[
L^2(\bbR) = H^+(\bbR) \oplus H^-(\bbR) .
\] 
On the circle we must be more careful, because the zero frequency mode contributes positive mass. However, if we restrict to the space $L^2_0(\bbR / \bbZ)$ of mean-zero $L^2$ functions, then 
\[
L^2_0(\bbR / \bbZ) = H^+(\bbR/\bbZ) \oplus H^-(\bbR/\bbZ).
\] 
Concordantly, for much of this paper we will assume that all our solutions to \eqref{BO} on the circle have mean 0. Because the \eqref{BO} flow preserves the mean of the data (since its right hand side is a complete derivative), this amounts to requiring the initial data to have mean 0. This assumption will be removed in the end by way of the Galilei transformation \eqref{galilei}.

The orthogonal Cauchy projections $C_{\pm} : L^2(\bbR) \to H^{\pm}(\bbR)$ and $C_{\pm} : L^2_0(\bbR/\bbZ) \to H^\pm(\bbR/\bbZ)$ are given by
\[
C_{\pm} f = \tfrac{1}{2} ( f \pm i H f) .
\]
Given a smooth, decaying function $q(t,x)$, we define operators $L_{\pm}, P_{\pm}$ by
\[
L_{\pm}(t) = \pm C_\pm \tfrac{1}{i} \partial_x  - C_\pm q(t) C_\pm ,
\]
\[
P_{\pm}(t) = \pm \tfrac{1}{i} C_\pm \partial_x^2 + 2( (C_\pm q_x(t))C_\pm - C_\pm q_x(t) - C_\pm q(t) C_\pm \partial_x )
\]

Because these operators leave $H^{\pm}$ (respectively) invariant, we are free to understand them to act on $L^2$ or on $H^\pm$. It was shown in \cite{wu} that $q(t)$ (mean 0 if on the circle) solves \eqref{BO} if and only if
\begin{equation} \label{lax-pair}
\frac{d}{dt} L_{\pm} = [L_{\pm}, P_{\pm}] .
\end{equation}
Let us restrict our attention to the action on $H^+$. Because of \eqref{lax-pair}, the \eqref{BO} flow preserves all the spectral properties of $L_+(t)$. Thus, formally, we expect the perturbation determinant (where the determinant is taken over $H^+$)
\[
\det ( (\kappa + L_+(t))R_\kappa) = \det({\rm id} - C_+ q(t) C_+ R_\kappa)
\]
to be preserved in time if $q$ solves \eqref{BO}. Here
\[
R_\kappa = C_+ (\kappa - i \partial_x)^{-1} C_+
\]
is defined by multiplication on the Fourier side by $\mathbbm{1}_{(0,\infty)}(\xi) (\kappa + \xi)^{-1}$. If $\kappa > 0$, this is a positive definite operator on $H^+$, and hence $\sqrt{R_\kappa}$ makes sense and the symbol of $\sqrt{R_\kappa}$ is the square root of that of $R_\kappa$. Its inverse $R_\kappa^{-1}$ also makes sense, albeit as an unbounded operator.

Taking a logarithm, we find
\begin{equation} \label{log-det}
- \log \det ( (\kappa + L_+(t))R_\kappa) = \sum_{\ell = 1}^\infty \frac{1}{\ell} \tr \big \{ \big ( C_+ q(t) C_+ R_\kappa \big)^\ell \big\} .
\end{equation}
It will be convenient to reformulate the above in terms of the operator
\[
A(\kappa; q) := \sqrt{R_\kappa} C_+ q C_+ \sqrt{R_\kappa}
\]
which depends linearly on $q$ and is self-adjoint when $q$ is real. Cycling the trace, we may rewrite \eqref{log-det} as
\[
\sum_{\ell = 1}^\infty \frac{1}{\ell} \tr \{ A(\kappa; q(t))^\ell \} .
\]
This quantity almost makes sense; however, $A(\kappa;q)$ is not a trace-class operator, even if $q$ is Schwartz. On the other hand, considered formally,
\[
\tr \{ A(\kappa; q(t)) \} = \tr \{ (\kappa + L_+(t)) R_\kappa  - {\rm id} \}
\]
ought to be preserved by the \eqref{BO} flow because of \eqref{lax-pair}. Thus we may have some confidence in dropping the $\ell = 1$ term to study the quantity 
\[
\alpha(\kappa ; q) := \sum_{\ell = 2}^\infty \frac{1}{\ell} \tr \{ A(\kappa; q)^\ell \} .
\]
As we shall see, this series makes sense if $q \in H^s$ for any $s > -\frac{1}{2}$ and $\kappa$ is sufficiently large. 


The crux of the method is to show, as the foregoing discussion suggests, that $\alpha(\kappa ; q)$ is conserved by the \eqref{BO} flow (section 2) and that it controls the relevant norm(s) of the solution (section 3). In our case and unlike in \cite{killip-visan-zhang}, the main term of $\alpha(\kappa; q)$ is not commensurate with any Sobolev norm of $q$. Since the $H^s$ norms for $s>- \frac{1}{2}$ are strictly stronger than the norm controlled by $\alpha(\kappa;q)$, the main theorem is recovered in section 3 from a kind of persistence of regularity.

\subsection{Notation and Preliminaries}

We write $A \lesssim B$ to mean that $A \leq C B$ for an absolute constant $C$; if the value of $C$ depends on parameters $a,b, \ldots$ then we will instead write $A \lesssim_{a,b,\ldots} B$. We write $A \lesssim B^{\gamma \pm}$ to mean that, for any $\eps > 0$, $A \lesssim_\eps B^{\gamma \pm \eps}$.

In this paper our conventions for the Fourier transform are
\[
\hat{f}(\xi) = \frac{1}{\sqrt{2 \pi}} \int_\bbR e^{- i x \xi} f(x) dx, \qquad \check{f}(x) = \frac{1}{\sqrt{2 \pi}} \int_\bbR e^{i x \xi} f(\xi) d \xi
\]
for functions on the line and
\[
\hat{f}(\xi) = \int_0^1 e^{- i x \xi} f(x) dx, \qquad \check{f}(x) = \sum_{\xi \in 2 \pi \bbZ} e^{i x \xi} f(\xi)
\]
for functions on the circle. We define
\[
\| f \|_{H^s(\bbR)}^2 = \int_\bbR (1 + |\xi|^2)^s |\hat{f}(\xi)|^2 d \xi , \qquad \| f \|_{H^s(\bbR / \bbZ)}^2 = \sum_{k \in 2 \pi \bbZ} (1 + |\xi|^2) |\hat{f}(\xi)|^2
\]
and let $H^s_0(\bbR / \bbZ)$ denote the subspace of $H^s(\bbR / \bbZ)$ functions with $\hat{f}(0) = 0$, i.e. mean zero.

Because our problem is translation-invariant, we may avoid any functional-analytic subtleties by working entirely on the Fourier side. If $T$ is a linear operator given on the Fourier side by
\[
\widehat{T \varphi}(\xi) = \int_\bbR K(\xi,\eta) \hat{\varphi}(\eta) d \eta
\]
then we may define the Hilbert-Schmidt norm of $T$ by
\[
\| T \|_{\frak{I}_2}^2 = \iint_{\bbR^2} |K(\xi, \eta)|^2  d \eta d \xi .
\]
Similarly, if $n \geq 2$ and $T_1, \ldots, T_n$ are Hilbert-Schmidt operators with Fourier kernels $K_1, \ldots, K_n$, then we say $T_1 \cdots T_n$ is trace class and define the trace
\[
\tr \{ T_1 \cdots T_n \} = \int_{\bbR^n} K_1(\xi_1, \xi_2) \cdots K_n(\xi_n, \xi_1) d \xi_1 \cdots d \xi_n . 
\]
In this formulation, cycling the trace amounts to an application of Fubini's theorem.

By the Cauchy-Schwarz inequality, $\alpha(\kappa;q)$ is a sub-geometric series with a common ratio $\lesssim \| A(q) \|_{\frak{I}_2}$. The following lemma gives sufficient conditions for this series to converge and submit to term-by-term differentiation and ensures that $\alpha(\kappa;q)$ is comparable to its first term.

\begin{lemma} \label{geometric-convergence}
Let $t \mapsto A(t)$ define a $C^1$ curve in $\frak{I}_2$. Suppose for some $t_0$ we have
\[
\|A(t_0)\|_{\frak{I}_2} < \frac{1}{3} .
\]
Then there is a closed interval $I$ containing $t_0$ on which the series
\[
\alpha(t) := \sum_{\ell = 2}^\infty \frac{1}{\ell} \tr \{ A(t)^\ell \}
\]
converges uniformly and defines a $C^1$ function which can be differentiated term by term:
\[
\frac{d}{dt} \alpha(t) = \sum_{\ell = 2}^\infty \tr \{ A(t)^{\ell - 1} \frac{d}{dt} A(t) \} .
\]
If $A(t)$ is self-adjoint, then
\[
\frac{1}{3} \| A(t)\|_{\frak{I}_2}^2 \leq \alpha(t) \leq \frac{2}{3} \| A(t) \|_{\frak{I}_2}^2 .
\]
\end{lemma}

For a proof of this lemma, see \cite{killip-visan-zhang}, Lemma 1.5.

\subsection{Acknowledgements} I am grateful for the guidance of my advisors, Rowan Killip and Monica Vi\c{s}an. In the preparation of this paper I received support from NSF grants DMS-1500707, DMS-1600942, and DMS-1763074.

\section{Conservation of the Perturbation Determinant}

In light of Lemma~\ref{geometric-convergence}, our first task is to understand $\|A(q(t))\|_{\frak{I}_2}$. Our next result is most conveniently formulated in terms of the linear operator $T_\kappa$ given by the Fourier multiplier 
\[
\widehat{T_\kappa f}(\xi) = \frac{\log(2 + |\xi| / \kappa)}{\sqrt{\kappa^2 + \xi^2}} \hat{f}(\xi) .
\]

\begin{theorem} \label{hilbert-schmidt-norm}
If $q \in H^s(\bbR)$ or $q \in H^s_0(\bbR / \bbZ)$ for $-\frac{1}{2} < s < 0 $, then for $\kappa \geq 1$
\[
\| A(\kappa; q) \|_{\frak{I}_2}^2 \sim \langle q, T_\kappa q \rangle \lesssim_s \kappa^{-1-2s} \| q \|_{H^s}^2 .
\]
\end{theorem}

\begin{proof}
We first consider the case of the line. We compute
\begin{align*}
\| A(\kappa; q)\|_{\frak{I}_2}^2 
	&= \int_{\xi \geq 0} \int_{\eta \geq 0} (\kappa + \xi)^{-1} (\kappa + \eta)^{-1} |\hat{q}(\xi - \eta)|^2 d \eta d \xi \\
	&= \int_{-\infty}^\infty \int_{\eta \geq \max (0, - \xi)} (\kappa + \xi + \eta)^{-1} (\kappa + \eta)^{-1} |\hat{q}(\xi)|^2 d \eta d \xi \\
	&= \int_0^\infty \frac{1}{\xi} \log \left( 1 + \frac{\xi}{\kappa} \right) |\hat{q}(\xi)|^2 d \xi - \int_{-\infty}^0 \frac{1}{\xi} \log \left( 1 - \frac{\xi}{\kappa} \right) |\hat{q}(\xi)|^2 d \xi \\
	&= \int_{- \infty}^\infty \frac{\log(1 + \frac{|\xi|}{\kappa})}{|\xi|} |\hat{q}(\xi)|^2 d \xi \\
	&\sim \int_{-\infty}^\infty \frac{\log(2 + \frac{|\xi|}{\kappa})}{\sqrt{\kappa^2 + \xi^2}} |\hat{q}(\xi)|^2 d \xi
\end{align*}
where the implicit constant in the last line is absolute. This proves the first inequality. The second inequality follows from the fact that
\[
\log(2 + |\xi| / \kappa) ( \kappa^2 + \xi^2)^{-1/2} \lesssim_s \kappa^{-1} \Big(1 + \Big(\frac{\xi}{\kappa}\Big)^2\Big)^{s} \leq \kappa^{-1-2s} (1 + \xi^2)^{s}
\]
for any $- \frac{1}{2} < s < 0$, $\kappa \geq 1$.

In the case $q \in H^s_0(\bbR / \bbZ)$, a similar computation to the above may be repeated, although the analogue of the third equality holds only within the bounds of multiplicative constants, rather than exactly.
\end{proof}

\begin{theorem} \label{constant-perturbation-determinant}
Let $q$ be a $C^0_t H^3_x \cap C^1_t H^1_x$ solution to \eqref{BO} on the line or the circle, having mean 0 if on the circle. For any $t \in \bbR$ and $s > - \frac{1}{2}$, there exists a constant $C = C(s)$ such that for all $\kappa \geq 1 + C \|q(t)\|_{H^s}^\frac{2}{1+2s}$,
\[
\frac{d}{dt} \alpha(\kappa; q(t)) = 0.
\]
\end{theorem}

\begin{proof}
We choose $C$ large enough that Theorem~\ref{hilbert-schmidt-norm} ensures that 
\[
\| A(\kappa ; q(t))\|_{\frak{I}_2} < \frac{1}{3}
\] 
whenever $\kappa \geq 1 + C \|q(t)\|_{H^s}^\frac{2}{1+2s}$. We then apply Lemma~\ref{geometric-convergence} to conclude that $\alpha(\kappa; q)$ converges on a neighborhood of $t$ and
\begin{align*}
\frac{d}{dt} \alpha(k;q(t)) 
	&= \sum_{\ell = 2}^\infty \tr \Big\{ A(\kappa,q)^{\ell - 1} A(\kappa;q_t) \Big\} \\
	&= \sum_{\ell = 2}^\infty \tr \Big \{ A(\kappa;q)^{\ell - 1} A(\kappa; -Hq'' + 2qq') \Big\} .
\end{align*}
By Theorem \ref{hilbert-schmidt-norm}, $A(\kappa;q)$ is a Hilbert-Schmidt operator, as is $A(\kappa; -H q'' + 2 q q')$ if $q \in H^3$, so we may cycle a copy of $A(\kappa;q)$ in the trace to obtain
\[
\frac{d}{dt} \alpha(\kappa;q(t)) = 
	- \sum_{\ell = 2}^\infty \tr \Big\{ A(\kappa;q)^{\ell - 2}A(\kappa;Hq'') A(\kappa;q) \Big\}
	+ \sum_{\ell = 2}^\infty  \tr \Big\{ A(\kappa;q)^{\ell - 1} A(\kappa;2qq') \Big\} ,
\]
which we rearrange slightly to give a telescoping series:
\begin{multline*}
\frac{d}{dt} \alpha(\kappa ; q(t)) = - \tr \Big\{ A(\kappa;q) A(\kappa;Hq'') \Big\} + \mbox{}
	 \\ + \sum_{\ell = 2}^\infty \left[ 2 \tr \Big\{ A(\kappa;q)^{\ell - 1} A(\kappa;qq') \Big\} -  \tr \Big\{ A(\kappa;q)^{\ell - 1} A(\kappa;Hq'') A(\kappa;q) \Big\} \right] .
\end{multline*}
Evidently it suffices to show that 
\begin{equation} \label{head}
\tr \Big\{ A(\kappa;q) A(\kappa;Hq'') \Big\} = 0
\end{equation} 
and
\begin{equation} \label{telescope}
2 \tr \Big\{ A(\kappa;q)^{\ell - 1} A(\kappa;qq') \Big\} =  \tr \Big\{ A(\kappa;q)^{\ell - 1} A(\kappa;Hq'') A(\kappa;q) \Big\} 
\end{equation}
for all $\ell \geq 2$.

To see \eqref{head}, we compute the trace directly on the line:
\begin{align*}
 \tr \{ A(\kappa;q) A & (\kappa;Hq'') \} \\
 	&= - \int_{\xi \geq 0} \int_{\eta \geq 0} (k + \xi)^{-1} \hat{q}(\xi - \eta) (k + \eta)^{-1} \hat{H}(\eta - \xi) (\eta - \xi)^2 \hat{q}(\eta - \xi) d \eta d \xi \\
 	&= - i \int_{\xi \geq 0} \int_{\eta \geq 0} \frac{\sgn(\xi - \eta) (\xi - \eta)^2}{(k+\xi)(k+\eta)} |\hat{q}(\xi - \eta)|^2 d \eta d \xi.
\end{align*}
This integral converges absolutely when $q \in H^2$. The integrand is odd with respect to $\xi = \eta$, so the integral evaluates to 0. The computation on the circle is similar.

To reduce the number of derivatives on $q$ in the right hand side of \eqref{telescope}, we require a Leibniz rule for the derivative operator $R_\kappa^{-1}$. If $f \in H^2$, we write
\[
C_+ f' C_+ = i C_+ [ C_+ (\kappa - i \partial_x) C_+ , f] C_+ = i C_+ [ R_\kappa^{-1}, f ] C_+
\]
and so, commuting $C_+$ and $R_\kappa$ as needed,
\begin{align*}
A(\kappa;q) A(\kappa;f') A(\kappa;q)
	&= i \sqrt{R_\kappa} C_+ q C_+ R_\kappa C_+ [R_\kappa^{-1} , f] C_+ R_\kappa C_+ q \sqrt{R_\kappa} \\
	&= i \sqrt{R_\kappa} C_+ q C_+ \big( f R_\kappa - R_\kappa f \big) C_+ q C_+ \sqrt{R_\kappa} \\
	&= i \sqrt{R_\kappa} C_+ q C_+ f C_+ \sqrt{R_\kappa} A(\kappa;q) - i A(\kappa;q) \sqrt{R_\kappa} C_+ f C_+ q C_+ \sqrt{R_\kappa}.
\end{align*}
Because $R_\kappa^{-1}$ is an unbounded operator, the first equality above holds only on the domain of $R_\kappa^{-1}$, which is a dense subset of $H^+$. However, $A(\kappa;q) \in \frak{I}_2$ and $\sqrt{R_\kappa} C_+ f C_+ g \sqrt{R_\kappa} \in \frak{I}_2$ when $f,g \in H^2$. This suffices to conclude
\begin{equation} \label{commutator-identity}
A(\kappa;q) A(\kappa;f') A(\kappa;q) = i \sqrt{R_\kappa} C_+ q C_+ f C_+ \sqrt{R_\kappa} A(\kappa;q) - i A(\kappa;q) \sqrt{R_\kappa} C_+ f C_+ q C_+ \sqrt{R_\kappa}
\end{equation}
with equality as operators on $H^+$.

Now we show \eqref{telescope}. We write 
\[H q'' = \tfrac{1}{i} q_+'' - \tfrac{1}{i} q_-'' = ( \tfrac{1}{i} q_+' - \tfrac{1}{i} q_-')' ,
\]
where $\varphi_{\pm}$ denotes the projection of $\varphi$ onto $H^{\pm}$. Letting $f = \tfrac{1}{i} q_+' - \tfrac{1}{i} q_-$ in \eqref{commutator-identity} , we find
\begin{align*}
\tr \{ &A(\kappa;q)^{\ell - 1} A(\kappa; H q'') A(\kappa;q) \} \\
	&= \tr \Big\{ A(\kappa;q)^{\ell - 2} \sqrt{R_\kappa} C_+ q C_+ q_+' C_+ \sqrt{R_\kappa} A(\kappa;q) \Big\} - \tr \Big\{ A(\kappa;q)^{\ell - 1} \sqrt{R_\kappa} C_+ q_+' C_+ q C_+ \sqrt{R_\kappa} \Big\} \\
	&\qquad - \tr \Big\{ A(\kappa;q)^{\ell - 2} \sqrt{R_\kappa} C_+ q C_+ q_-' C_+ \sqrt{R_\kappa} A(\kappa;q) \Big\} + \tr \Big\{ A(\kappa;q)^{\ell - 1} \sqrt{R_\kappa} C_+ q_-' C_+ q C_+ \sqrt{R_\kappa} \Big\} \\
	&= \tr \Big\{ A(\kappa;q)^{\ell - 1} \sqrt{R_\kappa} C_+ q C_+ q_+' C_+ \sqrt{R_\kappa} \Big\} - \tr \Big\{ A(\kappa;q)^{\ell - 1} \sqrt{R_\kappa} C_+ q_+' C_+ q C_+ \sqrt{R_\kappa} \Big\} \\
	&\qquad - \tr \Big\{ A(\kappa;q)^{\ell - 1} \sqrt{R_\kappa} C_+ q C_+ q_-' C_+ \sqrt{R_\kappa} \Big\} + \tr \Big\{ A(\kappa;q)^{\ell - 1} \sqrt{R_\kappa} C_+ q_-' C_+ q C_+ \sqrt{R_\kappa} \Big\} \\
	&=: A - B - C + D .
\end{align*}
We pass to the penultimate line above by cycling a copy of $A(\kappa;q)$ in two of the trace terms. Adding and subtracting $A+D$ yields
\[
\tr \{ A(\kappa;q)^{\ell - 1} A(\kappa; H q'') A(\kappa;q) \} = 2 (A+D) - A - B - C - D .
\]
We exploit some identities of the Cauchy projections in order to simplify the above expressions. If $f \in L^2(\bbR)$ or $f \in L^2_0 (\bbR / \bbZ)$, then $C_+ f_+ C_+ = f_+ C_+$ and $C_+ f_- C_+ = C_+ f_-$. Thus
\[
A = \tr \Big\{ A(\kappa;q)^{\ell - 1} \sqrt{R_\kappa} C_+ q q_+' C_+ \sqrt{R_\kappa} \Big\} , \quad D = \tr \Big\{ A(\kappa;q)^{\ell - 1} \sqrt{R_\kappa} C_+ q_-' q  C_+\sqrt{R_\kappa} \Big\} .
\]
Applying the identity $f_+ + f_- = f$, we find
\[
A + D = \tr \Big\{ A(\kappa;q)^{\ell - 1} A(\kappa;qq') \Big\} .
\]
Thus to show \eqref{telescope} and complete the proof of the theorem, it suffices to show $A + B + C + D = 0$. By the same identity, we may simplify
\[
A+C = \tr \Big\{ A(\kappa;q)^{\ell - 1} \sqrt{R_\kappa} C_+ q C_+ q' C_+ \sqrt{R_\kappa} \Big\} 
\]
and
\[
B+D = \tr \Big\{ A(\kappa;q)^{\ell - 1} \sqrt{R_\kappa} C_+ q' C_+ q C_+ \sqrt{R_\kappa} \Big\} .
\]
When $\ell \geq 3$, we apply the Leibniz identity
\[
i C_+ [ R_\kappa^{-1}, q C_+ q ] C_+ = C_+ q' C_+ q C_+ + C_+ q C_+ q' C_+
\]
and cycle a copy of $A(\kappa;q)$ in the trace to find $A+B+C+D = \tr \{X\}$, where
\begin{align*}
X
	&= i A(\kappa;q)^{\ell - 2} \sqrt{R_\kappa} C_+ [ R_\kappa^{-1}, q C_+ q ] C_+ \sqrt{R_\kappa} A(\kappa;q) \\
	&= i A(\kappa;q)^{\ell - 3} \sqrt{R_\kappa} C_+ q C_+ q C_+ q C_+ \sqrt{R_\kappa} A(\kappa;q) - i A(\kappa;q)^{\ell - 2} \sqrt{R_\kappa} C_+ q C_+ q C_+ q C_+ \sqrt{R_\kappa} .
\end{align*}
Because $\sqrt{R_\kappa} C_+ q C_+ q C_+ q C_+ \sqrt{R_\kappa} \in \mathfrak{I}_2$, we may substitute this into the trace and cycle a copy of $A(\kappa;q)$ to obtain
\[
A + B + C + D = \tr\{X\} = 0.
\]
In the case $\ell = 2$, we do not have two copies of $A(\kappa;q)$ to place around the commutator, so we cannot apply the Leibniz rule as an operator identity. Instead we apply the same idea at the level of the integrals:
\begin{align*}
(A + C) + (B + D) &= \int_{\xi \geq 0} \int_{\eta \geq 0} \int_{\nu \geq 0} \frac{i(\nu - \xi)}{(\kappa + \xi) (\kappa + \eta)} \hat{q}(\xi - \eta) \hat{q}(\eta - \nu) \hat{q}(\nu - \xi) d \nu d \eta d \xi \\
	& \qquad + \int_{\xi \geq 0} \int_{\eta \geq 0} \int_{\nu \geq 0} \frac{i(\eta - \nu)}{(\kappa + \xi) (\kappa + \eta)} \hat{q}(\xi - \eta) \hat{q}(\eta - \nu) \hat{q}(\nu - \xi) d \nu d \eta d \xi \\
	&= \int_{\xi \geq 0} \int_{\eta \geq 0} \int_{\nu \geq 0} \frac{i(\eta - \xi)}{(\kappa + \xi) (\kappa + \eta)} \hat{q}(\xi - \eta) \hat{q}(\eta - \nu) \hat{q}(\nu - \xi) d \nu d \eta d \xi \\
	&= i\int_{\xi \geq 0} \int_{\eta \geq 0} \int_{\nu \geq 0} \frac{1}{(\kappa + \xi)} \hat{q}(\xi - \eta) \hat{q}(\eta - \nu) \hat{q}(\nu - \xi) d \nu d \eta d \xi \\
	&\qquad - i \int_{\xi \geq 0} \int_{\eta \geq 0} \int_{\nu \geq 0} \frac{1}{(\kappa + \eta)} \hat{q}(\xi - \eta) \hat{q}(\eta - \nu) \hat{q}(\nu - \xi) d \nu d \eta d \xi .
\end{align*}
The above integrals converge by Cauchy-Schwarz. Cycling the variables $\xi \mapsto \nu \mapsto \eta \mapsto \xi$ in the second integral, we see that the two integrals in the last identity are equal. This completes the proof.
\end{proof}

Because $\alpha$ is comparable to its first term, as a corollary to this result we obtain uniform in time control of $\| A(\kappa; q(t)) \|_{\frak{I}_2}$.

\begin{corollary} \label{hs-control}
Let $s > -\frac{1}{2}$ and let $q$ be a $C^0_t H^3_x \cap C^1_t H^1_x$ solution to \eqref{BO} on the line or the circle, having mean 0 if on the circle. Then there exists a constant $C=C(s)$ such that for all $\kappa \geq 1 + C \| q(0)\|_{H^s}^\frac{2}{1+2s}$ ,
\[
\sup_{t \in \bbR} \| A(\kappa;q(t)) \|_{\frak{I}_2}^2 \leq 2 \| A(\kappa; q(0)) \|_{\frak{I}_2}^2 < \frac{1}{9}
\]
and therefore, by Theorem~\ref{hilbert-schmidt-norm},
\[
\langle q(t), T_{\kappa} q(t) \rangle \lesssim \langle q(0), T_{\kappa} q(0) \rangle.
\]
\end{corollary}

\begin{proof}

We may choose $C$ sufficiently large that $\| A(\kappa; q(0))\|_{\frak{I}_2}^2 < \frac{1}{18}$. By Lemma~\ref{geometric-convergence} and Theorem~\ref{constant-perturbation-determinant}, there exists a neighborhood $I$ of 0 on which
\begin{equation} \label{trace-control}
\| A(\kappa; q(t)) \|_{\frak{I}_2}^2 \leq 3\alpha(\kappa ; q(t)) = 3 \alpha(\kappa; q(0)) \leq 2 \| A(\kappa ; q(0)) \|_{\frak{I}_2}^2 < \frac{1}{9} .
\end{equation}
Since $\|A(\kappa ; q(t))\|_{\frak{I}_2} < \frac{1}{3}$, Lemma 1 implies that \eqref{trace-control} is an open condition, and the theorem follows by a continuity argument.
\end{proof}

\section{Conservation of Norms}
Because of the logarithmic factor, $\langle q , T_\kappa q \rangle$ is not commensurate with any $H^s$ norm of $q$; it behaves like $\| q\|_{H^{-1/2}}^2$ at frequencies $\lesssim \kappa$ and like $\| \log (|\nabla|) \langle \nabla \rangle^{-1/2} q \|_{L^2}^2$ at frequencies $\gg \kappa$. This difficulty is avoided if we ``build'' $\|q\|_{H^s}$ for $-\frac{1}{2} < s < 0$ one frequency scale at a time, using the contribution of $\langle q , T_\kappa q \rangle$ at the frequency scale $\kappa$ where it behaves like a pure Sobolev norm.

This is naturally expressed in terms of the Besov norms
\[
\| f \|_{B^{s,2}_r} = \left( \| \hat{f}(\xi)| \|_{L^2(|\xi| \leq 1)}^r + \sum_{N > 1} N^{rs} \| \hat{f}(\xi) \|_{L^2(N \leq |\xi| < 2N)}^r \right)^{1/r}
\]
where the sum is taken over dyadic $N = 2,4,8,\ldots$ and with the usual interpretation in the case $r = \infty$. The following lemma (the analogue of Lemma 3.2 in \cite{killip-visan-zhang}) relates this norm to (the leading term of) $\alpha(k ; q)$.

\begin{lemma} \label{besov-building}
Fix $-\frac{1}{2} < s < 0$, $1 \leq r \leq \infty$, $\kappa_0 \geq 1$. For any $H^2$  function $f$,
\begin{equation} \label{build-by-scales-right}
\| f \|_{B^{s,2}_r}^r 
	\lesssim \sum_{N \in 2^{\bbN}} N^{rs} \big( \kappa_0 N \langle f, T_{\kappa_0N} f \rangle \big)^{r/2}
\end{equation}
and
\begin{equation} \label{build-by-scales-left}
\sum_{N \in 2^\bbN} N^{rs} \big( \kappa_0 N \langle f, T_{\kappa_0 N} f \rangle \big)^{r/2}
	\lesssim_s \kappa_0^{-rs} \| f \|_{B^{s,2}_r}^r .
\end{equation}
\end{lemma}

\begin{proof}
The inequality \eqref{build-by-scales-right} follows easily from the estimate
\[
\| \hat{f}(\xi) \|_{L^2(|\xi| \leq N)}^2 \leq \frac{2}{\log 2} \int \frac{\kappa_0 N \log ( 2 + \frac{|\xi|}{\kappa_0 N})}{\sqrt{\kappa_0^2 N^2 + \xi^2}} |\hat{f}(\xi)|^2 d \xi .
\]

To control the other direction, we decompose
\begin{align*}
\int & \frac{\kappa_0 N \log ( 2 + \frac{|\xi|}{\kappa_0 N})}{\sqrt{\kappa_0^2 N^2 + \xi^2}} |\hat{f}(\xi)|^2 d \xi \\
	&\leq \log(3) \| \hat{f}(\xi)\|_{L^2(|\xi| \leq 1)}^2 + \sum_{M \in 2^\bbN} \frac{\kappa_0 N \log(2 + \frac{2M}{\kappa_0 N})}{\sqrt{\kappa_0 N^2 + M^2}} \| \hat{f}(\xi)\|_{L^2(M < |\xi| \leq 2M)}^2 \\
	&\leq \left( \sqrt{\log(3)} \| \hat{f}(\xi)\|_{L^2(|\xi| \leq 1)} + \sum_{M \in 2^\bbN} \left( \frac{ \kappa_0 N \log(2 + \frac{2M}{\kappa_0 N})}{\sqrt{\kappa_0^2 N^2 + M^2}} \right)^{1/2} \| \hat{f}(\xi)\|_{L^2(M < |\xi| \leq 2M)} \right)^2.
\end{align*}
This shows that the left-hand side of \eqref{build-by-scales-left} is bounded by 
\[
\left\| \sqrt{\log(3)} N^s \| \hat{f}(\xi) \|_{L^2(|\xi| \leq 1)} + \sum_{M \in 2^\bbN} \left( \frac{ \kappa_0 N^{1 + 2s} M^{-2s} \log(2 + \frac{2M}{\kappa_0 N})}{\sqrt{\kappa_0^2 N^2 + M^2}} \right)^{1/2} M^s \| \hat{f}(\xi)\|_{L^2(M < |\xi| \leq 2M)} \right\|_{\ell^r(N \in 2^\bbN)}^r
\]
which reduces our task to estimating the operator norm of a certain $\ell^r \to \ell^r$ matrix. To do this, we apply Schur's test. The row sums of this operator are bounded by
\[
\sqrt{\log(3)} N^s + \sum_{M \in 2^\bbN} \left( \frac{ \kappa_0 N^{1 + 2s} M^{-2s} \log(2 + \frac{2M}{\kappa_0 N})}{\sqrt{\kappa_0^2 N^2 + M^2}} \right)^{1/2} \lesssim_s 1 + \kappa_0^{-s}
\]
uniformly in $N$, while the column sums are bounded by 
\[
\sum_{N \in 2^\bbN} \sqrt{\log(3)} N^s \lesssim_s 1, \qquad \sum_{N \in 2^\bbN}  \left( \frac{ \kappa_0 N^{1 + 2s} M^{-2s} \log(2 + \frac{2M}{\kappa_0 N})}{\sqrt{\kappa_0^2 N^2 + M^2}} \right)^{1/2} \lesssim_s \kappa_0^{-s}
\]
uniformly in $M$. Note that to make these estimates we require the condition $- \frac{1}{2} < s < 0$. This proves \eqref{build-by-scales-left}.
\end{proof}

Our main result now follows easily from the foregoing lemma and Corollary \ref{hs-control}.

\begin{theorem} \label{main-theorem}
Let $q$ be a $C^0_tH^3_x \cap C^1_tH^1_x$ solution to \eqref{BO} on the line or the circle and let $- \frac{1}{2} < s < 0$, $1 \leq r \leq \infty$. Then
\[
(1 + \| q(0)\|^{\frac{2}{1+2s}}_{B^{s,2}_r})^{s} \sup_{t \in \bbR} \| q(t) \|_{B^{s,2}_r} \lesssim_{s,r} \| q(0) \|_{B^{s,2}_r}
	\lesssim_{s,r} (1 + \|q(0)\|^{\frac{2}{(1+2s)^2}}_{B^{s,2}_r})^{-s} \inf_{t \in \bbR} \| q(t)\|_{B^{s,2}_r} .
\]
\end{theorem}

\begin{proof}
On the circle, we first assume that $q$ has mean 0. By H\"older's inequality, we have an embedding $B^{s_1,2}_r \hookrightarrow B^{s_2,2}_2 = H^{s_2}$ for any $s_2 < s_1$. Let 
\[
\kappa_0 = 1 + C \| q(0) \|_{B^{s,2}_r}^\frac{2}{1+2s} \gtrsim_s 1 + C \| q(0)\|_{H^{s-}}^\frac{2}{1+2s}
\] 
for a sufficiently large constant $C$, so that we may apply Corollary \ref{hs-control}. Then, for any time $t$, Lemma~\ref{besov-building} implies
\begin{align*}
\| q(t) \|_{B^{s,2}_r}^r 
	&\lesssim \sum_{N \in 2^\bbN} N^{rs} \big( \kappa_0 N \langle q(t), T_{\kappa_0 N} q(t) \rangle \big)^{r/2} \\
	&\lesssim \sum_{N \in 2^\bbN} N^{rs} \big( \kappa_0 N \langle q(0), T_{\kappa_0 N} q(0) \rangle \big)^{r/2} \\
	&\lesssim_s \kappa_0^{-rs} \| q(0) \|_{B^{s,2}_r}^r \\
	&= (1 + \| q(0)\|_{B^{s,2}_r}^{\frac{2}{1+2s}} )^{-rs} \| q(0) \|_{B^{s,2}_r}^r .
\end{align*}
This proves the first inequality. By time translation symmetry, we then also obtain
\[
\| q(0) \|_{B_r^{s,2}} \lesssim_s (1 + \| q(t)\|_{B^{s,2}_r}^{\frac{2}{1+2s}} )^{-rs} \| q(t) \|_{B^{s,2}_r}
\]
and applying the first inequality to the quantity in parentheses produces the second inequality.

To remove the mean zero assumption on the circle, we employ Galilean invariance: if $q$ solves \eqref{BO}, then so does
\begin{equation} \label{galilei}
\tilde{q}(t,x) = q(t,x + 2 \mu t) + \mu .
\end{equation}
The estimate 
\[
\| \tilde{q}(t)\|_{B^{s,2}_r(\bbR / \bbZ)}^2 \sim \| q(t)\|_{B^{s,2}_r(\bbR/\bbZ)}^2 + \mu^2
\]
then implies the general theorem.
\end{proof}

\end{document}